\documentclass[11pt]{article}
\usepackage{amsmath, amsthm}
\usepackage{latexsym}
\usepackage{graphicx, psfrag}
\makeatletter
\newfont{\footsc}{cmcsc10 at 8truept}
\newfont{\footbf}{cmbx10 at 8truept}
\newfont{\footrm}{cmr10 at 10truept}
\makeatother
\newtheorem{theorem}{\bf Theorem}
\newtheorem{proposition}{\bf Proposition}
\newtheorem{lemma}{\bf Lemma}

\newtheorem{corollary}{\bf Corollary}

\begin{document}
\title{On Normal Variance-Mean Mixtures}

\author{Yaming Yu\\
\small Department of Statistics\\[-0.8ex]
\small University of California\\[-0.8ex] 
\small Irvine, CA 92697, USA\\[-0.8ex]
\small \texttt{yamingy@uci.edu}}

\date{}
\maketitle

\begin{abstract}
Normal variance-mean mixtures encompass a large family of useful distributions such as the generalized 
hyperbolic distribution, which itself includes the Student $t$, Laplace, hyperbolic, normal inverse Gaussian, and
variance gamma distributions as special cases.  We study shape properties of normal variance-mean mixtures, in both
the univariate and multivariate cases, and determine conditions for unimodality and log-concavity of the density 
functions.  This leads to a short proof of the unimodality of all generalized hyperbolic densities.  We also interpret 
such results in practical terms and discuss discrete analogues. 


\end{abstract}

\section{Introduction}
A univariate normal variance-mean mixture (Barndorff-Nielsen, Kent and S{\o}rensen 1982) is the distribution of 
\begin{equation}
\label{mix}
Y=\mu +\beta X + \sigma \sqrt{X}Z
\end{equation}
where $X$ and $Z$ are independent scalar random variables, $Z\sim {\rm N}(0, 1)$, $X$ has a density (the mixing density) supported on $(0,\infty)$, and $-\infty<\mu, \beta<\infty,\ \sigma>0$ are constants.  Equivalently, a normal variance-mean mixture is the distribution of a Brownian motion with drift stopped at an independent random time.  Normal variance-mean mixtures encompass a large family of distributions commonly used in many applied fields.  A prominent example is the generalized hyperbolic (GH) distribution (Barndorff-Nielsen, 1977).  This distribution has an unnormalized density of the form ($y\in \mathbf{R}$) 
\begin{equation}
\label{dens}
{\rm gh}(y; \mu, \lambda, \alpha, \beta, \delta)\propto  e^{\beta(y-\mu)} \left(\delta^2 + (y-\mu)^2\right)^{(\lambda-1/2)/2}
K_{\lambda-1/2} \left(\alpha \sqrt{\delta^2+(y-\mu)^2}\right),
\end{equation}
where $K_\nu(z)$ denotes the modified Bessel function of the second kind (Abramowitz and Stegun 1972).  The allowable parameter values are $\mu\in \mathbf{R}$ and 
\begin{align*}
\delta\geq 0,&\quad |\beta|<\alpha,\quad {\rm if}\ \lambda >0;\\
\delta> 0,&\quad |\beta|<\alpha,\quad {\rm if}\ \lambda =0;\\
\delta> 0,&\quad |\beta|\leq \alpha,\quad {\rm if}\ \lambda <0.
\end{align*}
The density (\ref{dens}) arises as the density of $Y$ in (\ref{mix}) if we let $\sigma=1$ and let the density of $X$ be the generalized inverse Gaussian (GIG; J{\o}rgensen 1982) 
\begin{equation}
\label{gig}
{\rm gig}(x; \lambda, \chi, \psi)= \frac{(\psi/\chi)^{\lambda/2}}{2K_\lambda(\sqrt{\psi\chi})} x^{\lambda-1} \exp\left[-\frac{1}{2}(\chi x^{-1} +\psi x)\right],\quad x>0. 
\end{equation}
Parameters other than the common $\lambda$ are related by $\chi=\delta^2$ and $\psi=\alpha^2-\beta^2$. 

The density (\ref{dens}) includes the Student $t$, Laplace, hyperbolic, normal inverse Gaussian, and variance gamma densities as special cases.  These are important in modeling financial data; see, e.g., Eberlein and Keller (1995), Barndorff-Nielsen (1997), Eberlein (2001), and Chen, H\"{a}rdle and Jeong (2004).  The appearance of modified Bessel functions, however, makes them nontrivial for both theoretical analysis and numerical computation.  We refer to Protassov (2004) for computational strategies in parameter estimation. 

This paper is motivated by a basic property of the GH density, namely its unimodality.  A function $f(x)$ on $\mathbf{R}$ is unimodal, if there exists $x_* \in \mathbf{R}$ such that $f(x)$ increases on $(-\infty, x_*]$ and decreases on $[x_*,\infty)$.  Besides being inherently interesting, unimodality results are also useful.  Chebyshev's bounds for tail probabilities can be sharpened considerably if the distribution is known to be unimodal (Sellke and Sellke 1997).  Shape properties such as unimodality are also important in random variate generation (Devroye 1986). 

All GH densities are unimodal.  Although this can be verified analytically in some cases, it is certainly not obvious from the formula (\ref{dens}).  The only general proof of unimodality we know is based on two deep results: (i) GH distributions are self-decomposable (Halgreen 1979; Shanbhag and Sreehari 1979; Sato 2001) and (ii) self-decomposable distributions are unimodal (Yamazato 1978; Steutel and van Harn 2003).  Conceivably there may exist a direct proof based on modified Bessel functions inequalities (see Yu 2011 for a similar situation).  In this paper, however, we take another approach and obtain a unimodality condition for a general normal variance-mean mixture.  The GH case follows as a corollary. 

Our main result (Theorem~\ref{thm1}) says that unimodality is preserved by forming univariate normal variance-mean mixtures.  The same holds for log-concavity.  A nonnegative function $f(x),\ x\in\mathbf{R},$ is log-concave if 
$$f(\theta x+(1-\theta)y)\geq \left[f(x)\right]^\theta\left[f(y)\right]^{1-\theta}\quad {\rm for\ all}\quad x,y\in \mathbf{R},\ \theta\in (0, 1).$$
Log-concavity implies unimodality and is often referred to as ``strong unimodality'' (Dharmadhikari and Joag-dev, 1988).  We also consider log-convexity.  A positive function $f$ on an interval $I$ is log-convex if $\log f$ is convex on $I$. 

\begin{theorem}
\label{thm1}
Suppose $X$ and $Y$ are related by (\ref{mix}) with densities $g$ and $f$ respectively. 
\begin{itemize}
\item[i]
If $g$ is unimodal, then so is $f$. 
\item[ii]
If $g$ decreases on $(0,\infty)$, or $\beta=0$, then the only mode of $f$ is at $\mu$. 
\item[iii]
If $g$ is log-concave, then so is $f$. 
\item[iv]
If $g$ is log-convex on $(0,\infty)$, then $f$ is log-convex on each of $(-\infty, \mu)$ and $(\mu, \infty)$. 
\end{itemize}
\end{theorem}

Theorem~\ref{thm1} has clear practical implications.  According to part (i), one must incorporate a multimodal mixing distribution in order for a normal variance-mean mixture model to capture multimodality in the data.  This contrasts with other shape properties; skewness, for example, depends on the parameter $\beta$ and is not entirely inherited from the mixing distribution.  On the other hand, log-concave densities have no heavier than exponential tails.  Part (iii) therefore indicates that normal variance-mean mixtures inherit their heavy tails from the mixing distributions.  That is, one needs to choose a heavy-tailed mixing distribution in (\ref{mix}) in order to model heavy tails in the data. 

In the case of the GH distribution, Theorem~\ref{thm1} specializes to Corollary~\ref{coro1}.  We obtain a short proof of unimodality and a simple criterion for log-concavity for GH densities. 

\begin{corollary}
\label{coro1}
Let $f(y)$ denote the GH density given by (\ref{dens}).  
\begin{itemize}
\item[i]
The density $f(y)$ is unimodal.
\item[ii]
The mode of $f(y)$ is at $\mu$ iff either $\beta=0$ or $\delta=0,\, 0<\lambda\leq 1$. 
\item[iii]
The density $f(y)$ is log-concave iff $\lambda\geq 1$. 
\item[iv]
The density $f(y)$ is log-convex on each of $(-\infty, \mu)$ and $(\mu, \infty)$ iff $\delta=0$ and $0<\lambda\leq 1$. 
\end{itemize}
\end{corollary}

Theorem~\ref{thm1} is related to the following result for normal mean mixtures (Dharmadhikari and Joag-dev, 1988; Bertin and Theodorescu, 1995).  Suppose $Y|X \sim {\rm N}(X, \sigma^2)$ and $X$ has density $g$.  Then, denoting the marginal density of $Y$ by $f$, we have (i) if $g$ is unimodal then so is $f$; (ii) if $g$ is log-concave then so is $f$.  While these two hold by basic convolution properties of log-concave functions, the proof of Theorem~\ref{thm1} is more complicated.  A key tool is the theory of total positivity (Karlin 1968); see Section~3 for details. 

We also consider the multivariate case.  A $p$-variate normal variance-mean mixture is the distribution of 
\begin{equation} 
\label{mmix}
Y=\mu + \beta X + \sqrt{X} AZ
\end{equation}
where $\mu, \beta$ are $p\times 1$ vectors, $A$ is a full-rank $p\times p$ matrix, $X$ is a scalar random variable supported on $(0, \infty)$,
and $Z$ is a $p\times 1$ standard normal vector independent of $X$.  Owing to properties of the multivariate normal,
such mixtures are closed under marginalization and linear transformations.  They have the following feature regardless of the shape of the mixing density  ($\|\cdot\|$ denotes the Euclidean norm). 
\begin{proposition}
\label{spindle}
Suppose $\beta\neq 0$ in (\ref{mmix}).  Then for each $t\in \mathbf{R}$, the density of 
$Y$, denoted by $f(y)$, has ellipsoidal contours on the hyperplane 
$$\mathcal{H}_t\equiv \mu + \beta t + \left\{z\in\mathbf{R}^p:\ z^\top (AA^\top)^{-1}\beta=0 \right\}.$$ 
Specifically, $f(y)$ depends only on $\|A^{-1}(y-\mu-\beta t)\|$ for $y\in \mathcal{H}_t$. 
\end{proposition}
For example, if $A$ is the identity matrix, then $f(y)$ is spherically symmetric on any hyperplane orthogonal to $\beta$ (one may say $f(y)$ has ``spindle-like'' contours, $\beta$ being the direction of the center). 

Theorem~\ref{thm2} extends Theorem~\ref{thm1} to the multivariate case.  The notion of 
log-concavity is trivially extended from the univariate case.  Unimodality is more delicate, however, 
and there exist several alternative definitions (Dharmadhikari and Joag-dev 1988).  Theorem~\ref{thm2} uses a rather weak form of unimodality. 

\begin{theorem}
\label{thm2}
Suppose $X$ and $Y$ are related by (\ref{mmix}) with densities $g(x)$ and $f(y)$ on $x\in (0,\infty)$ and $y\in \mathbf{R}^p$ respectively.  Let $g^*(x)=x^{-(p-1)/2} g(x)$.  
\begin{itemize}
\item[i]
If $g^*(x)$ is unimodal, then $f(y)$ has only one local maximum, which lies on the line $y=\mu+\beta t,\ t\in\mathbf{R}$.
\item[ii]
If $g^*(x)$ decreases on $(0,\infty)$, or $\beta=0$, then the only local maximum of $f(y)$ is at $y=\mu$. 
\item[iii]
If $g^*(x)$ is log-concave, then so is $f(y)$. 
\item[iv]
If $g^*(x)$ is log-convex on $(0,\infty)$, then $f(\mu + bt)$ is log-convex on $t\in (0, \infty)$ for all nonzero $b\in \mathbf{R}^p$. 
\end{itemize}
\end{theorem}

An important case is the multivariate generalized hyperbolic (MGH) distribution, which arises when $X$ in (\ref{mmix}) has a GIG distribution.  The MGH distribution is further closed under conditioning, that is, if $Y=(Y_1, Y_2)^\top$ has an MGH distribution then so does $Y_1$ conditional on $Y_2$.  This allows us to strengthen the unimodality part of Theorem~\ref{thm2}.  We say a function $f(y)$ on $\mathbf{R}^p$ has convex contours, if the set $\{y\in\mathbf{R}^p:\ f(y)\geq c\}$ is convex for all $c\in \mathbf{R}$.  Obviously, all local maxima of a density having convex contours must be global maxima. 

\begin{corollary}
\label{coro2}
Suppose the density of $X$ in (\ref{mmix}) is given by (\ref{gig}), and let $f(y)$ denote the density of $Y$. 
\begin{itemize}
\item[i]
The density $f(y)$ has convex contours.
\item[ii]
The mode of $f(y)$ is at $\mu$ iff $\beta=0$ or $\chi=0,\, 0<\lambda\leq (p+1)/2$. 
\item[iii]
The density $f(y)$ is log-concave iff $\lambda\geq (p+1)/2$. 
\item[iv]
The function $f(\mu+bt)$ is log-convex on $t\in (0,\infty)$ for all nonzero $b\in\mathbf{R}^p$ iff $\chi=0$ and $0<\lambda\leq (p+1)/2$. 
\end{itemize}
\end{corollary} 

Theorem~\ref{thm2} and Corollary~\ref{coro2} are proved by reducing the problem to a univariate one and then applying Theorem~\ref{thm1}; see Section~3.  We discuss discrete analogues in Section~2. 

\section{Discrete Analogues}
There exist analogous results for mixtures of discrete distributions (Karlin 1968; Bertin and Theodorescu 1995).  In the univariate case, discrete unimodality and log-concavity are defined in obvious parallel with their continuous counterparts.  We mention the following result for Poisson mixtures. 

\begin{theorem}
\label{thm3}
Let $X$ be a positive random variable with Lebesgue density $g(x)$.  Given $X$, let $Y$ have a Poisson$(X)$ distribution, and denote the marginal probability mass function of $Y$ by $f(y)$.  If $g(x)$ is unimodal on $x\in (0,\infty)$ then so is $f(y)$ on $y=0, 1, \ldots$.  The same holds if ``unimodal'' is replaced by ``decreasing,'' ``log-concave,'' or ``log-convex''.
\end{theorem}
The unimodality part of Theorem~\ref{thm3} is a classical result of Holgate (1970); see also Bertin and Theodorescu (1995).  The decreasing and log-convexity parts are obtained by Steutel and van Harn (2003).  All four parts can be proved using the same approach as in our proof of Theorem~\ref{thm1}. 

The practical implications of Theorem~\ref{thm3} are similar to those of Theorem~\ref{thm1}.  Roughly, when fitting a Poisson mixture model, one needs to choose a multi-modal (respectively, heavy-tailed) mixing distribution in order to capture multi-modality (respectively, heavy tails) in the data.  Note that ``heavy'' means ``heavier than the exponential''.

Sichel (1982) studies a generalized inverse Gaussian Poisson (GIG-P) distribution in an economics context.  This is a Poisson mixture as in Theorem~\ref{thm3} where the mixing density $g(x)$ is a generalized inverse Gaussian (\ref{gig}).  When $\chi=0$ the resulting mixture is a negative binomial.  Corollary~\ref{coro3} states unimodality and log-concavity results analogous to those of Corollary~\ref{coro1}.  We omit the parallel derivations. 

\begin{corollary}
\label{coro3}
In the setting of Theorem~\ref{thm3} let $g(x)$ be the GIG density given by (\ref{gig}). 
\begin{itemize}
\item[i]
The GIG-P probability mass function, $f(y)$, is unimodal.
\item[ii]
The function $f(y)$ is decreasing iff $f(0)\geq f(1)$. 
\item[iii]
The function $f(y)$ is log-concave iff $\lambda\geq 1$. 
\end{itemize}
\end{corollary} 
The analogy is not complete, however.  Characterizing log-convexity for the GIG-P family seems to be an open problem, for which Theorem~\ref{thm3} only yields a trivial sufficient condition. 

\section{Derivation of Main Results}
For convenience, we say that a function $h(x)$ has property $S$ on an interval $I$, if $h(x)$ has at most two sign changes on $x\in I$ and, in the case of two changes, the sign pattern is $-, +, -$.  (The number of sign changes is counted discarding zero terms.)  Lemma~\ref{lem1} characterizes unimodality (respectively, log-concavity) in terms of how $h(x)$ crosses all horizontal lines (respectively, exponential curves).  
\begin{lemma}
\label{lem1}
A function $h(x)$ is unimodal on an interval $I$ iff $h(x)- c$ has property $S$ on $I$ for all $c\in \mathbf{R}$;
a nonnegative $h(x)$ is log-concave on $I$ iff $h(x)-a e^{b x}$ has property $S$ on $I$ for all $a>0,\, b\in \mathbf{R}$.
\end{lemma} 

A kernel $k(x, y)$ on $x\in I$ and $y\in J$ ($I, J$ are intervals) is totally positive (TP) if for each $m\geq 1$ and all $x_1< x_2<\cdots<x_m,\ y_1<y_2<\cdots<y_m,\ x_i\in I,\ y_j\in J$, the determinant of the $m\times m$ matrix with $k(x_i, y_j)$ as the $(i, j)$ entry is nonnegative.  The kernel $k(x, y)$ is strictly totally positive (STP) if such determinants are strictly positive.  We rely on a variation-diminishing property of TP kernels; see Karlin (1968) for the general theory. 

\begin{lemma}
\label{lem2}
Suppose $k(x, y)$ is a TP kernel on $x\in I,\ y\in J$, and $h(x)$ is a function that has $n$ sign changes on $I$.  Then, assuming the integral is absolutely convergent, $q(y)\equiv \int_I k(x, y) h(x)\, {\rm d}x$ has at most $n$ sign changes on $y\in J$.  Moreover, if $q(y)$ has exactly $n$ sign changes, then its sign pattern as $y$ increases on $J$ is the same as that of $h(x)$ as $x$ increases on $I$.  In particular, if $h(x)$ has property $S$ on $I$ then $q(y)$ has property $S$ on $J$. 
\end{lemma}
The idea of using Lemmas~\ref{lem1} and \ref{lem2} to establish unimodality, log-concavity or log-convexity comes from Propositions 3.1 and 3.2 in Chapter 1 of Karlin (1968).  Karlin's results concern monotonicity and concavity, but the method works for unimodality, log-concavity and log-convexity after some modifications. 

\begin{proof}[Proof of Theorem~\ref{thm1}]
If $\beta=0$ then $Y$ in (\ref{mix}) is unimodal with a unique mode at $y=\mu$ for arbitrary $g$.  Hence parts (i)-(ii) hold.  Parts (iii)--(iv) will follow from the $\beta\neq 0$ case by taking limits.  Let us consider $\beta \neq 0$.  After a linear transformation we may assume $\mu=0,\, \beta>0$ and $\sigma=1$.  The density of $Y$ becomes 
\begin{align*}
f(y) &=\int_0^\infty k(x, y) g(x)\, {\rm d}x,\quad -\infty<y<\infty,\\
k(x, y) &\equiv \frac{1}{\sqrt{2\pi x}} \exp\left[\frac{-(y-\beta x)^2}{2x}\right].
\end{align*}
Two properties of the kernel $k(x, y)$ are easily established: 
\begin{itemize}
\item[(a)]
$k(x, y)$ is STP in $x>0,\, y\geq 0$;
\item[(b)]
we have the identity 
\begin{equation}
\label{eqn0}
\int_0^\infty k(x, y)\, {\rm d}x = \frac{1}{\beta},\quad y\geq 0.
\end{equation}
\end{itemize}
Part (a) is obvious after writing $k(x, y)=u(x) v(y) \exp\left[- y^2/(2x)\right]$ (see Karlin 1968, p.\ 18).  Part (b) is a calculus exercise.  Alternatively, (b) holds because the GIG density (\ref{gig}) with $\lambda =1/2$ integrates to one. 

If $g(x)$ is unimodal, then by Lemma~\ref{lem1}, $g(x)- c\beta$ has property $S$ on $x\in (0, \infty)$ for arbitrary $c$.  By (\ref{eqn0}) we have  
\begin{equation}
\label{eqn1}
f(y)- c = \int_0^\infty k(x,y) \left[g(x) - c\beta\right]\, {\rm d}x,\quad y\geq 0.
\end{equation}
By Lemma~\ref{lem2}, $f(y)-c$ has property $S$ on $y\in [0, \infty)$.  We conclude that $f(y)$ is unimodal on $y\in [0, \infty)$ since $c$ is arbitrary.  For $y<0$ we may differentiate inside the integral and directly obtain 
\begin{equation}
\label{deri}
f'(y)=\int_0^\infty \left(\beta-\frac{y}{x}\right) k(x, y) g(x)\, {\rm d}x >0.
\end{equation}
Thus $f(y)$ is unimodal with a nonnegative mode overall.  By invoking strict total positivity, we can show that the mode is unique (see Karlin 1968, Chapter 5, Theorem~3.1).  This proves part (i). 

Similarly, if $g(x)$ is decreasing, then Lemma~\ref{lem2} and (\ref{eqn1}) show that $f(y)-c$ changes signs at most once, from $+$ to $-$, on $y\in [0,\infty)$.  As $c\in\mathbf{R}$ is arbitrary, this means $f(y)$ decreases on $[0, \infty)$, which implies that the mode is at $y=0$.  This proves part (ii). 

For parts (iii)-(iv) let us define 
$$l(y)\equiv \log f(y) - \log c - (\beta-\gamma)y$$ 
for $c>0$ and $\gamma\in \mathbf{R}$.  As long as $\gamma\neq 0$, we have the following generalization of (\ref{eqn1}) 
\begin{align}
\label{eqn2}
f(y)- c \exp\left(\beta y -|\gamma y|\right) &= \int_0^\infty k(x,y) h(x)\, {\rm d}x,\quad y\in \mathbf{R};\\
\nonumber
h(x) &\equiv g(x) - c|\gamma| \exp\left(\frac{\beta^2-\gamma^2}{2} x\right).
\end{align}
Suppose $g(x)$ is log-concave.  By Lemma~\ref{lem1}, $h(x)$ has property $S$ on $x\in (0, \infty)$.  By Lemma~\ref{lem2}, the right hand side of (\ref{eqn2}) has property $S$ on $y\in [0, \infty)$. 

Suppose $\gamma>0$.  Then $l(y)$ has property $S$ on $[0,\infty)$ because it has the same sign pattern as the left hand side of (\ref{eqn2}).  For $y<0$ we apply (\ref{deri}) to obtain 
$$l'(y)= \frac{f'(y)}{f(y)} - \beta +\gamma > \gamma >0.$$
If the sign pattern of $l(y)$ on $[0, \infty)$ is $-, +, -$, then $l(y)\leq 0$ for $y<0$ because of monotonicity.  Hence the overall sign pattern of $l(y)$ on $(-\infty, \infty)$ is still $-, +, -$.  Other cases are similar.  It follows that $l(y)$ has property $S$ on $y\in (-\infty, \infty)$. 

Suppose $\gamma<0$.  We can similarly show that $l(y)$ decreases on $y\in [0, \infty)$ and has property $S$ on $y\in (-\infty, 0]$.  It follows that $l(y)$ has property $S$ for arbitrary $c>0$ and $\gamma\in \mathbf{R}$.  By Lemma~\ref{lem1}, $f(y)$ is log-concave, which proves part (iii). 

Finally, suppose $g(x)$ is log-convex on $(0,\infty)$.  Parallel to part (iii), we can show that if $\gamma>0$, then $-l(y)$ has property $S$ on 
$y\in (0, \infty)$.  If $\gamma<0$ then $-l(y)$ increases on $(0,\infty)$.  Hence $-l(y)$ has property $S$ on $y\in (0,\infty)$ for arbitrary $\gamma\in \mathbf{R}$.  It follows that $f(y)$ is log-convex on $(0,\infty)$.  The $(-\infty, 0)$ case is similar.  Hence part (iv) holds. 
\end{proof} 

{\bf Remark}.  In Theorem~\ref{thm1}, under the stronger assumption that $x^{-1/2} g(x)$ is log-concave, we have an alternative proof of the log-concavity of $f(y)$.  The key is Pr\'{e}kopa's theorem, which states that if a joint Lebesgue density is log-concave then so are the marginals.  The density of $(X, Y)$ is 
$$u(x, y)\propto x^{-1/2} g(x) \exp\left[-\frac{(y-\beta x)^2}{2x}\right],\quad x>0,\ y\in\mathbf{R}.$$
It is easy to verify that $(y-\beta x)^2/(2x)$ is jointly convex in $x>0,\ y\in\mathbf{R}$.  Hence $u(x,y)$ is jointly log-concave.  By Pr\'{e}kopa's theorem, $f(y)=\int_0^\infty u(x, y)\, {\rm d}x$ is log-concave in $y\in \mathbf{R}$.  This argument, however, does not work under the weaker assumption that $g(x)$ is log-concave.  See Yu (2009) for a similar situation. 

To prove Corollary~\ref{coro1}, the following facts concerning modified Bessel functions are used (see Abramowitz and Stegun 1972, Chapter 9).  We have 
\begin{align} 
\label{bess1}
K_\nu(z)&=K_{-\nu}(z);\\
\lim_{z\to 0} z^\nu K_{\nu}(z) &= 2^{\nu-1} \Gamma(\nu),\quad \nu>0;\\
\label{bess2}
K'_\nu(z) &= -\frac{\nu}{z} K_\nu(z) - K_{\nu-1}(z).
\end{align}

\begin{proof}[Proof of Corollary~\ref{coro1}]
Let $g(x)$ denote the GIG density given by (\ref{gig}) with $\chi=\delta^2,\ \psi=\alpha^2-\beta^2$.  Then $g(x)$ is unimodal, and so is $f(y)$ by Theorem~\ref{thm1}.  Hence part (i) holds.  For parts (ii)-(iv) we assume $\mu=0,\ \beta\geq 0$.  If $\delta=0,\ 0<\lambda\leq 1$, then the mixing density $g(x)$ is decreasing, and by Theorem~\ref{thm1}, $f(y)$ has a unique mode at $0$.  Conversely, if $\beta>0$ and $\delta>0$ or $\beta>0,\ \delta=0$ and $\lambda>1$, then direct calculation using (\ref{dens}) and (\ref{bess1})--(\ref{bess2}) yields $f'(0)/f(0) =\beta >0$.  Hence the mode is strictly positive.  This proves part (ii).  If $\lambda\geq 1$, then the GIG density is log-concave, and so is $f(y)$ by Theorem~\ref{thm1}.  If $\lambda <1$ then $f(y)$ cannot be log-concave because as $y\to \pm \infty$, it behaves like $|y|^{\lambda-1} e^{(\beta\mp \alpha)y}$ up to a multiplicative constant (Barndorff-Nielsen and Bl{\ae}sild 1981).  This proves part (iii).  Finally, if $\delta=0$ and $0<\lambda\leq 1$, then $g(x)$ is log-convex.  By Theorem~\ref{thm1}, $f(y)$ is log-convex on each of $(0,\infty)$ and $(-\infty, 0)$.  Conversely, if $f(y)$ is log-convex on each of $(0,\infty)$ and $(-\infty, 0)$, then it increases on $(-\infty, 0)$ and decreases on $(0,\infty)$, because $f(y)\to 0$ as $|y|\to \infty$.  By part (ii), we must have either $\beta=0$ or $\delta=0<\lambda\leq 1$.  If $\beta=0,\ \delta>0$ or $\beta=\delta=0,\ \lambda>1$, then by (\ref{dens}) and (\ref{bess1})--(\ref{bess2}), we have $f'(0)=0$.  By log-convexity $f'(y)\geq 0$ for $y>0$, which is again a contradiction.  Therefore we must have $\delta=0$ and $0<\lambda\leq 1$.  This proves part (iv). 
\end{proof}

The rest of this section proves Proposition~\ref{spindle}, Theorem~\ref{thm2} and Corollary~\ref{coro2}.  Let us henceforth assume $\mu$ is all zero and $A$ is the identity matrix in (\ref{mmix}); the general results follow after a linear transformation.  The density of $Y$ is 
\begin{equation}
\label{mdens}
f(y)=\int_0^\infty (2\pi x)^{-p/2} g(x) \exp\left(-\frac{\|y -\beta x\|^2}{2x}\right)\, {\rm d}x, 
\end{equation}
where $g(x)$ again denotes the mixing density.

\begin{proof}[Proof of Proposition~\ref{spindle}]
For $y\in \mathcal{H}_t$ we can write $y=\beta t + z$ such that $z^\top \beta =0$.  Then (\ref{mdens}) leads to
\begin{equation}
\label{hyper}
f(y) = \int_0^\infty (2\pi x)^{-p/2} g(x) \exp\left[-\frac{\|\beta\|^2 (t-x)^2 + \|z\|^2}{2x}\right]\, {\rm d}x.
\end{equation}
Hence $f(y)$ depends only on $\|z\|$ and in fact decreases in $\|z\|$, yielding spherical contours on each $\mathcal{H}_t$ (a linear transformation makes them ellipsoidal). 
\end{proof}

\begin{proof}[Proof of Theorem~\ref{thm2}]
If $\beta$ is all zero, then $f(y)$ is symmetric with a single mode at zero for arbitrary $g$.  Let us assume $\beta\neq 0$ and consider $f(y)$ on the hyperplane $\mathcal{H}_t$ as in (\ref{hyper}).  Evidently, if $y_*\in \mathcal{H}_t$ is a local mode, and $y_*=\beta t+z_*,$ then $\|z_*\|=0$ by (\ref{hyper}).  Thus all local modes of $f(y)$ lie on the line $y=\beta t,\ t\in\mathbf{R}$.  We have  
$$f(\beta t)= \int_0^\infty (2\pi x)^{-p/2} g(x) \exp\left[-\frac{\|\beta\|^2 (t- x)^2}{2x}\right] \, {\rm d}x.$$
If $g^*(x)= x^{-(p-1)/2} g(x)$ is unimodal, then by Theorem~\ref{thm1}, $f(\beta t)$ is unimodal in $t\in \mathbf{R}$ with a unique mode.  (Although Theorem~\ref{thm1} is stated in terms of a proper mixing density, we note here that whether $\int_0^\infty g^*(x)\, {\rm d}x<\infty$ is not essential; we only need $f(\beta t)<\infty$ for $t\neq 0$, which is easily verified.)  That is, $f(y)$ has a single local maximum on the line $y=\beta t,\ t\in\mathbf{R}$.  This proves part (i).  Moreover, if $g^*(x)$ is decreasing, then by Theorem~\ref{thm1}, the mode of $f(\beta t)$ is at $t=0$.  Thus part (ii) holds. 

For part (iii), let us consider $f(a+bt)$ as a function of $t\in \mathbf{R}$ for $a, b\in\mathbf{R}^p$ such that $\|b\|=1$.  After some algebra, (\ref{mdens}) yields 
\begin{equation}
\label{fabt}
\exp\left[ \left(\|\beta\|-b^\top \beta\right)t\right] f(a+bt) = C\int_0^\infty x^{-1/2} \tilde{g}(x) \exp\left[-\frac{(t + a^\top b -\|\beta\| x)^2}{2x} \right]\, {\rm d}x, 
\end{equation}
where $C$ is a constant not depending on $t$, and
\begin{equation}
\label{lcc}
\tilde{g}(x)=g^*(x) \exp\left[\frac{(a^\top b)^2-\|a\|^2}{2x}\right].
\end{equation}
If $g^*(x)$ is log-concave on $(0,\infty)$, then so is $\tilde{g}(x)$ because $|a^\top b|\leq \|a\|\|b\|=\|a\|$.  By Theorem~\ref{thm1}, the right hand side of (\ref{fabt}) is log-concave in $t\in \mathbf{R}$, and so is $f(a+bt)$.  The claim holds since $a,\, b$ are arbitrary. 

Finally, suppose $g^*(x)$ is log-convex.  Let us consider $a=0$ in (\ref{fabt}) and (\ref{lcc}).  Then $\tilde{g}(x)$ is log-convex.  By Theorem~\ref{thm1}, the right hand side of (\ref{fabt}) is log-convex in $t\in (0,\infty)$, and so is $f(bt)$.  This proves part (iv). 
\end{proof}

\begin{proof}[Proof of Corollary~\ref{coro2}]
Let $a, b\in \mathbf{R}^p$ such that $b\neq 0$.  Then $f(a+bt)$ is proportional to a univariate GH density due to closure properties of the MGH density over linear transformations and conditioning.  By Corollary~\ref{coro1}, $f(a+bt)$ is unimodal in $t\in \mathbf{R}$, which implies convex contours as $a, b$ are arbitrary.  This proves part (i).  Proofs of the other parts parallel those of Corollary~\ref{coro1}.  Consider part (iii) for example.  If $\lambda\geq (p+1)/2$, then $x^{-(p-1)/2}g(x)$ is log-concave, and by Theorem~\ref{thm2}, so is $f(y)$.  Conversely, if $f(y)$ is log-concave, then by considering $f(a+bt)$ with fixed $a,\, b$, Corollary~\ref{coro1} yields $\lambda\geq (p+1)/2$.  Hence the log-concavity claim holds. 
\end{proof} 



\begin{thebibliography}{10}
\bibitem{AS72}
M. Abramowitz and I. Stegun, {\it Handbook of Mathematical Functions
with Formulas, Graphs, and Mathematical Tables}, 9th ed. New York,
NY: Dover Press, 1972. 

\bibitem{B77}
O. E. Barndorff-Nielsen, ``Exponentially decreasing distributions for the logarithm of particle size,'' 
{\it Proc. Roy. Soc. London A}, vol. 353, pp. 401–-419, 1977. 

\bibitem{B97}
O. E. Barndorff-Nielsen, ``Normal inverse Gaussian distributions and stochastic
volatility modelling,'' {\it Scandinavian Journal of Statistics}, vol. 24, pp. 1--13, 1997. 

\bibitem{BB81}
O. E. Barndorff-Nielsen and P. Bl{\ae}sild, ``Hyperbolic distributions and ramifications: contributions to theory and application,'' In C. Taillie, G. Patil, and B. Baldessari (Eds.), {\it Statistical Distributions in Scientific Work,} vol. 4, pp. 19--44, 1981, Dordrecht: Reidel. 

\bibitem{BKS}
O. E. Barndorff-Nielsen, J. Kent and M. S{\o}rensen, ``Normal variance-mean mixtures and z-distributions,'' {\it International Statistical Review}, vol. 50, pp. 145--159, 1982. 

\bibitem{BT}
E. Bertin and R. Theodorescu, ``Preserving unimodality by mixing,'' {\it Statist. Probab. Lett.}, vol. 25, pp. 281--288, 1995. 

\bibitem{CHJ}
Y. Chen, W. H\"{a}rdle and S. Jeong, ``Nonparametric risk management with generalised hyperbolic distributions,'' 
{\it Technical Report}, CASE, Humboldt University, 2004. 

\bibitem{D86}
L. Devroye, {\it Non-Uniform Random Variate Generation}, Springer-Verlag, New York, 1986.

\bibitem{DJ}
S. W. Dharmadhikari and K. Joag-dev, {\it Unimodality, Convexity, and Applications}, Academic Press, New York, 1988. 

\bibitem{E01}
E. Eberlein, ``Application of generalized hyperbolic L\'{e}vy motions to finance,'' in: O. E.
Barndorff-Nielsen, T. Mikosch and S. Resnick, Eds., {\it L\'{e}vy Processes: Theory and
Applications}, Birkh\"{a}user, Boston, pp. 319–-337, 2001. 

\bibitem{EK}
E. Eberlein and U. Keller, ``Hyperbolic distributions in finance,'' {\it Bernoulli}, vol. 1, pp. 281--299, 1995. 

\bibitem{H79}
C. Halgreen, ``Self-decomposability of the generalized-inverse-Gaussian and hyperbolic distributions,'' {\it Z. Wahrsch. verw. Gebiete}, vol. 47,
pp. 13–-17, 1979. 

\bibitem{H}
P. Holgate, ``The modality of some compound Poisson distributions,'' {\it Biometrika}, vol. 57, pp. 666--667, 1970. 


\bibitem{J82}
B. J{\o}rgensen, ``Statistical properties of the generalized inverse Gaussian distribution,'' In {\it Lecture Notes in Statistics}, vol. 9, Heidelberg: Springer, 1982. 
 
\bibitem{K68}
S. Karlin, {\it Total Positivity}. Stanford: Stanford Univ. Press, 1968. 

\bibitem{P04}
R. S. Protassov, ``EM-based maximum likelihood parameter estimation for multivariate generalized hyperbolic distributions with fixed $\lambda$,'' {\it Statist. Comput.}, vol. 14, pp. 67--77, 2004. 

\bibitem{S01}
K. Sato, ``Subordination and self-decomposability,'' {\it Statist. Probab.
Lett.}, vol. 54, pp. 317–-324, 2001. 

\bibitem{SS97}
T. M. Sellke and S. H. Sellke, ``Chebyshev inequalities for unimodal distributions,'' {\it The American Statistician},
vol. 51, pp. 34--40, 1997. 

\bibitem{SS79}
D. N. Shanbhag and M. Sreehari, ``An extension of Goldie's result and further results in infinite divisibility,'' {\it Z. Wahrsch. verw. Gebiete}, vol. 47, pp. 19–-25, 1979. 

\bibitem{S82}
H. S. Sichel, ``Repeat-buying and the generalized inverse Gaussian-Poisson distribution,'' {\it Appl. Statist.}, vol. 31, pp. 193--204, 1982. 

\bibitem{SV03}
F. W. Steutel and K. van Harn, {\it Infinite Divisibility of Probability Distributions on the Real Line}, New York: Marcel-Dekker, 2003. 

\bibitem{Y78}
M. Yamazato, ``Unimodality of infinitely divisible distribution functions of class $L$,'' {\it Ann. Probab.}, vol. 6, pp. 523–-531, 1978. 

\bibitem{Y10}
Y. Yu, ``Some stochastic inequalities for weighted sums,'' {\it Preprint}, 2009, arXiv:0910.0544. 

\bibitem{Y11}
Y. Yu, ``On log-concavity of the generalized Marcum Q function,'' {\it Preprint}, 2011, arXiv:1105.5762. 

\end{thebibliography}
\end{document}